\documentclass[12pt]{article}
\usepackage{changepage}
\usepackage{mathtools}
\usepackage{url}
\usepackage{amsmath}
\usepackage{enumitem}
\usepackage{amsmath,bm}
\usepackage{amsfonts, amsmath, amssymb, amsthm, mathrsfs}
\usepackage{hyperref}
\usepackage{tikz}
\usepackage[labelsep=space]{caption}
\usepackage{float}

\newtheorem{theorem}{Theorem}
\newtheorem{claim}{Claim}
\newtheorem{lemma}{Lemma}
\newtheorem{observation}{Observation}

\begin{document}

%
\title{On the Ramsey number of a graph obtained by attaching a pendant to a path with length 1 modulo 4}
\author{Gabriel Razgon\thanks{King's College London. Email: \tt{gabriel.razgon@kcl.ac.uk}}}

\date{}
\maketitle

\abstract{In this paper, for odd $n$, we consider the graph $T_n$ with vertex set $\{x_1,....,x_n,x_{n+1}\}$ and edge set $\{x_{1}x_{2},x_{2}x_{3},...,x_{n-1}x_{n}\} \cup \{x_{\frac{n+1}{2}}x_{n+1}\}$ and prove that the Ramsey number $R(T_n,T_n)$ is equal to $\frac{3n+1}{2}$ for $n \in [1]_{4}$.}
\newline
\newline
\section{Introduction:}  


For each odd positive $n$, if $x_1,....,x_n$ are the consecutive vertices of the path $P_n$, let $T_n$ be the graph obtained from $P_n$ by adding a new vertex $w$ to $V(P_n)$ and
attaching an edge from the midpoint $x_{(n+1)/2}$ of $P_n$ to $w$.
\newline
If $P$ is a copy of $P_n$ contained in $T_n$, we denote $P$ as a \emph{spine} of $T_n$ and we denote the vertex $w$ of $V(T_n) \setminus V(P)$ as a \emph{pendant} of $T_n$. 
\newline
\newline
Recall that for simple graphs $H_1,...,H_k$, the Ramsey number $R(H_1,....,H_k)$ is the smallest integer $l$ such that for each $k$-edge-colouring $f$
of the clique $K_{l}$, if $G_1,...,G_k$ are the graphs induced by the respective colour classes of $f$, there exists $j \in \{1,...,k\}$ such that $H_j$ is a subgraph of $G_j$.
\newline
By the classical theorem of Ramsey, $R(H_1,...,H_k)$ is well-defined. 
See (\cite{Notes}) and (\cite{Book}) for proofs of this theorem and many other results of structural graph theory.
\newline
For a graph $H$, $R(H,H)$ is known as the Ramsey number of $H$.
\newline
\newline
Consider the class of trees $T$ such that there exists a non-empty path $P(T)$ in $T$ (a \emph{spine} of $T$) where each vertex of $V(T) \setminus V(P(T))$ is a leaf in $T$,
which is a class of graphs that contains all $T_n$.  The elements of this class of graphs are known as the caterpillars. 
\newline
\newline
The Ramsey numbers of all graphs in some classes of caterpillars such as the brooms (\cite{Broom1}, \cite{Broom2}), the caterpillars with vertex set 
\newline
$\{a_1,...,a_{2p}\} \cup \{v_1,...,v_n\} \cup \{w_1,w_2\}$ for $p \geq n \geq 2$ and edge set 
$\{a_{j}a_{j+1}:j \in \{1,...,2p-1\}\} \cup \{a_{1}v_j : j \in \{1,...,n\}\} \cup \{a_{2p}w_j : j \in \{1,2\}\} (\cite{Caterpillar1})$ and the caterpillars with vertex set $\{v_{0},...,v_{n-1}\}$ and edge set $\{v_{0}v_{1},...,v_{0}v_{n-4}\} \cup \{v_{1}v_{n-3},v_{1}v_{n-2},v_{2}v_{n-1}\}$ for $n \geq 6$ (\cite{Caterpillar2}) have been determined.
\newline
However, the problem of determining the Ramsey number of a general caterpillar is currently open, and there are few results on the exact Ramsey numbers of non-broom caterpillars $T$
with a spine $P(T)$ with endpoints $u$ and $v$ where at least one vertex of $P(T)$ is not adjacent to $u$ or $v$ and is adjacent to a leaf in $T$, which is a class of caterpillars that contains $T_n$ for $n \geq 5$. 
\newline
\newline
In light of the above review, in this 3-section paper, we prove the following theorem, where we determine the exact Ramsey number of the caterpillar $T_n$ for all $n \in [1]_{4}$ in terms of $n$.
\newline
Our theorem extends the result in \cite{Asymptote}, where it is proved that there exists a positive constant $o$ such that our theorem for $T_n$ holds for all $n>o$.
\newline
\newline
\begin{theorem}
The Ramsey number $R(T_n,T_n)$ is equal to $\frac{3n+1}{2}$ for all $n \in [1]_{4}$.
\newline
\newline
\end{theorem}

We prove Theorem 1 as follows.
Note that $n \in [1]_{4}$ throughout the paper. 
\newline
In Section \ref{sec:2}, we prove that $R(T_n,T_n) \leq \frac{3n+1}{2}$ for $n>5$,
and in Section \ref{sec:3}, we prove that $R(T_n,T_n) \geq \frac{3n+1}{2}$ for all $n$ and $R(T_n,T_n)=\frac{3n+1}{2}$ for $n \leq 5$.
\newline
\newline
The following notation will be used throughout the proof.
\newline
For a tuple $T$ of vertices of a graph $H$, if there exists a path in $H$ with tuple of consecutive vertices $T$, let $P_{T,H}$ denote this path.
For any tuple $T$,  let $T(j)$ be the $j$th entry of $T$, and for any two tuples $A=(a_1,..,a_{n_1})$, $B=(b_1,...,b_{n_2})$, we define $A+B$ to be $(a_1,...,a_{n_1},b_1,...,b_{n_2})$ .

\section{$\mathbf{R(T_n,T_n) \leq \frac{3n+1}{2}}$ for $\mathbf{n>5}$.} \label{sec:2}

Let $r$ be $\frac{3n+1}{2}$, let $f$ be an arbitrary 2-edge-colouring of $K_r$ with colour classes $C_1,C_2$ and for each $x \in \{1,2\}$, let $G_x$ be the graph $(V(K_r),C_x)$. Since $R(P_n,P_n)$ is known to be $\lfloor \frac{3n-2}{2} \rfloor$ for $n \geq 2$ (\cite{PathRamseyNumber}),
it follows that some $G_j$ contains a copy of $P_{n+1}$. 
\newline
\newline
Let $P$ be a longest path in $G_j$ (which has length at least $n+1$) and assume without loss of generality that $j=1$. 
Since the statement is proved if $T_n$ is a subgraph of $G_1$, henceforth assume that $T_n$ is not isomorphic to any subgraph of $G_1$. 
\newline
We prove that $T_n$ is a subgraph of $G_2$, which implies that $K_r$ contains a monochromatic copy of $T_n$.
This implies that $R(T_n,T_n) \leq r$ by definition. 
\newline
\newline
Let $T=(v_1,.....,v_{|V(P)|})$ be the tuple of consecutive vertices of $P$, and let $Q$ be $V(K_r) \setminus V(P)$.
\newline
Let $P'$ be $P_{(v_1,...,v_{n+1}),G_1}$ and let $z$ be $|V(P)|-(n+1)$, which is at most $(n-1)/2$ as $|V(G_1)| \leq \frac{3n+1}{2}$. 
\newline
We initially prove the following auxiliary lemmas in Subsection \ref{subsec:2.1}.
\newline
\newline
In Lemma \ref{lem:1}, we prove that for the tuple 
$V=(v_{n+2},v_{(n+2)-(n+1)/2})+
\newline
+(v_{n+3}, v_{(n+3)-(n+1)/2})+,...,+(v_{n+z},v_{n+z-(n+1)/2})+(v_{n+z+1},v_{n+z+1-(n+1)/2})$,
the $2z$-path $P_{V,G_2}$ exists if $z \neq 0$. 
\newline 
\newline
In Lemma \ref{lem:2}, we prove the following.
Let $P_{L,G_1}$ be a subpath of $P$ where $L=(l_1,...,l_k)$ and let $C$ be a subset of $Q$ with $|C| \geq \lfloor (k+1)/2 \rfloor$  such that if $k \geq 3$, at least two elements of $C$ are adjacent to $l_1$ and if $k<3$, some element of $C$ is adjacent to $l_1$.
Then, if $k'=\lfloor (k+1)/2 \rfloor$, $P_{M,G_2}$ exists for some tuple  $M=(m_1,....,m_{2k'})$ of elements of $V(K_r)$ such that $m_j \in V(P_{L,G_1})$ iff $j$ is odd and $m_j \in C$ otherwise
and $m_1 \in \{l_1,l_2,l_3\}$. 
\newline
\newline
Subsequently, in Subsection \ref{subsec:2.2}, we consider several cases with respect to the value of $z=|V(P)|-(n+1)$ and apply the above lemmas to each of these cases to construct a copy of $T_n$ in $G_2$.
\newline
\newline
\subsection{Two auxiliary lemmas.}\label{subsec:2.1}
\begin{lemma}\label{lem:1}
If $z \neq 0$ and  
$V$ is the tuple 
$(v_{n+2},v_{(n+2)-(n+1)/2})+
\newline
+(v_{n+3},v_{(n+3)-(n+1)/2})+,...,+
(v_{n+z},v_{n+z-(n+1)/2})+(v_{n+z+1},v_{n+z+1-(n+1)/2})$,
the $2z$-path $P_{V,G_2}$ exists.
\end{lemma}
\begin{proof}
Suppose that for some $j \in \{n+2,....,(n+1)+z\}$, $v_j$ and $v_{j-(n+1)/2}$ are not adjacent in $G_2$. 
Then, by definition, $e=v_{j}v_{j-(n+1)/2}$ is an edge of $G_1$. 
Let $k$ be $j-(n+1)/2$.
Since $j \geq (n+2)$, each element of $\{v_{k-1},...,v_{k-(n-1)/2}\} \cup \{v_{k+1},...,v_{k+(n-1)/2}\}$ is well-defined.
Therefore, if $L=(v_{k-(n-1)/2},....,v_k,.....,v_{k+(n-1)/2})$, the path $P_{L,G_1}$ is a well-defined subpath of the path $P$ in $G_1$
(which has consecutive vertices $v_1,....,v_{(n+1)+z}$), which adheres to the following constraints.
\newline
$P_{L,G_1}$ has odd length $n$, midpoint $v_k$ and does not contain $v_j=v_{k+(n+1)/2}$. 
\newline
Therefore, the subgraph $G'=(V(P_{L,G_1}) \cup \{v_j\}, E(P_{L,G_1}) \cup \{v_{k}v_{j}\})$ of $G_1$ is a copy of $T_n$ with spine $P_{L,G_1}$ and 
pendant $v_j$, which is a contradiction, as by assumption, $G_1$ contains no copy of $T_n$.
\newline
\newline
Therefore, $v_{j}v_{j-(n+1)/2} \in E(G_2)$. 
If $j<n+z+1$, $v_{j+1}$ is a vertex of $V(P)$ such that $v_{j+1}$ is not contained in the path $P_{L,G_1}$ defined in the previous paragraph. 
Therefore, we obtain a contradiction if $v_{j-(n+1)/2}v_{j+1} \in G_1$.
\newline
This implies that for all $j<n+z+1$ with $j \geq n+2$, 
\newline
$\{v_{j}v_{j-(n+1)/2},v_{j-(n+1)/2}v_{j+1}\} \subseteq E(G_2)$. 
\newline
Observe also that since $z \leq (n-1)/2$, the sets $\{v_{n+2},...,v_{n+z+1}\}$ and $\{v_{n+2-(n+1)/2},....,v_{n+z+1-(n+1)/2}\}$ are disjoint.
Therefore, the $2z$-path $P_{V,G_2}$ exists. 
\newline
\newline
\end{proof}
\begin{lemma}\label{Lemma 2}\label{lem:2}
Let  $P_{L,G_1}$ be a subpath of $P$ where $L=(l_1,...,l_k)$. Let $C$ be a subset of $Q$ with $|C| \geq \lfloor (k+1)/2 \rfloor$ such that if $k \geq 3$, at least two elements of $C$ are adjacent to $l_1$ and if $k<3$, some element of $C$ is adjacent to $l_1$.
Then, if $k'=\lfloor (k+1)/2 \rfloor$, $P_{M,G_2}$ exists for some tuple  $M=(m_1,....,m_{2k'})$ of elements of $V(K_r)$ such that $m_j \in V(P_{L,G_1})$ iff $j$ is odd and $m_j \in C$ otherwise
and $m_1 \in \{l_1,l_2,l_3\}$. 
\end{lemma}
\begin{proof}
Assume w.l.o.g that $(l_1,...,l_k)$ is a subtuple of $(v_1,...,v_{|V(P)|})$, as the lemma holds by symmetry if $(l_k,...,l_1)$ is a subtuple of $(v_1,...,v_{|V(P)|})$.
\newline
We prove the lemma by induction on $k$. 
\newline
For $k=1$ and $k=2$, if $m_2$ is an element of $C$ adjacent to $l_1$,
$P_{(l_1,m_2),G_2}$ exists and adheres to the given constraints.
\newline
For $k=3$ and $k=4$, let $m_2$ and $m_4$ be two elements of $C$ adjacent to $l_1$.
\newline
If $m_2$ is adjacent to both $l_2$ and $l_3$ in $G_1$,
since $P_{L,G_1}$ is a subpath of the path $P$ of $G_1$, which has consecutive vertices $v_1,...,v_{|V(P)|}$,
the path $P_{(v_1,....,l_2)+(m_2)+(l_{3},...,v_{|V(P)|}),G_1}$ is a path in $G_1$ of longer length than the $|V(P)|$-path $P$,
which is a contradiction as $P$ is a longest path of $G_1$. 
\newline
Therefore, in $G_2$, $m_2$ is adjacent to $l_t$ for some $t \in \{2,3\}$,
implying that $P_{(l_t,m_2,l_1,m_4),G_2}$ exists. 
\newline
Since the path $P_{(l_t,m_2,l_1,m_4),G_2}$ has length $4=2*\lfloor (4+1)/2 \rfloor=2*\lfloor (3+1)/2 \rfloor$,
it adheres to the given constraints.
\newline
\newline
For $k>4$, assume that the lemma holds for all smaller positive integers than $k$.
For $k>4$, let $L'$ be $(l_1,...,l_{c})$ for $c=k-4$. 
\newline
\newline
By the induction assumption, $P_{M',G_2}$ exists for some tuple $M'=(m_1,...,m_{2*\lfloor (c+1)/2 \rfloor})$ of elements of $V(K_r)$ such that $m_j \in V(P_{L',G_1})$ if $j$ is odd and $m_j \in C$ otherwise and $m_1 \in \{l_1,l_2,l_3\}$. 
\newline
\newline
Let $Y$ be $P_{M',G_2}$ and let $C'$ be $V(Y) \cap C$. By definition, $|C'|=\lfloor (c+1)/2 \rfloor$. 
Observe that $|C|-|C'| \geq \lfloor (k+1)/2 \rfloor-\lfloor (k-3)/2 \rfloor$, which is at least 2. 
Therefore, there exist distinct elements $u,v$ of $C \setminus C'$. Let $w \notin \{u,v\}$ be the endpoint $m_{2*\lfloor (c+1)/2 \rfloor}$ of $Y$ in $C'$. 
We derive the lemma from the following claim. 
\newline
\newline
\begin{claim}\label{clm:1}
If $c_1,c_2$ are two distinct elements of $\{u,v,w\}$, in $G_2$,
$c_1$ and $c_2$ have a common neighbour in $\{l_{k-3},l_{k-2},l_{k-1},l_{k}\}$.
\end{claim}
\begin{proof}

We begin with the following observation. 

\begin{observation}\label{obs:1}
For each such $c_j$, if $c_j$ is adjacent to both $l_t$ and $l_{t+1}$ in $G_1$ for some $t \in \{k-3,k-2,k-1\}$,
since $P_{L,G_1}$ is a subpath of the path $P$ of $G_1$, which has consecutive vertices $v_1,...,v_{|V(P)|}$,
the path $P_{(v_1,....,l_t)+(c_j)+(l_{t+1},...,v_{|V(P)|}),G_1}$ is a longer path in $G_1$ than the $|V(P)|$-path $P$,
which is a contradiction as $P$ is a longest path of $G_1$. 
\newline
\newline
\end{observation}
Therefore, in $G_2$, each $c_j$ is adjacent to some element of $\{l_t,l_{t+1}\}$ for each $t \in \{k-3,k-2,k-1\}$.
Suppose that in $G_2$, $c_1$ and $c_2$ have no common neighbour in $\{l_{k-3},l_{k-2}, l_{k-1},l_k\}$.  
Assume without loss of generality that $(c_2,l_{k-2}) \in E(G_1)$. By Observation \ref{obs:1}, $(c_2,l_{k-3}) \in E(G_2)$,
which implies by assumption that $(c_1,l_{k-3}) \in E(G_1)$.
\newline
\newline
Since in $G_2$, $c_2$ has a neighbour in $\{l_{k-2},l_{k-1}\}$, $(c_2,l_{k-1}) \in E(G_2)$. 
Therefore, by assumption, $(c_1,l_{k-1}) \in E(G_1)$, and since in $G_2$,
$c_1$ has a neighbour in $\{l_{k-1},l_k\}$, $(c_1,l_k) \in E(G_2)$, implying by assumption that
$(c_2,l_k) \in E(G_1)$.
\newline
\newline
Therefore, by assumption, $\{(c_2,l_{k-2}),(c_2,l_k),(c_1,l_{k-3}),(c_1,l_{k-1})\} \subseteq E(G_1)$,
implying that the path $P_{(v_1,.....,l_{k-3})+(c_1,l_{k-1},l_{k-2},c_2)+(l_k,...,v_{|V(P)|}),G_1}$ is a well-defined path in $G_1$
of longer length than the $|V(P)|$-path $P$, which is a contradiction as $P$ is a longest path of $G_1$. 
\newline
\newline
Therefore, in $G_2$, $c_1$ and $c_2$ have a common neighbour in $\{l_{k-3},l_{k-2},l_{k-1},l_k\}$,
proving the claim. 
\end{proof}
By Claim \ref{clm:1}, the set $X$ of common neighbours of $u$ and $v$ in $\{l_{k-3},l_{k-2},l_{k-1},l_k\}$ in $G_2$ is non-empty.
\newline
If $|X|>1$, then if in $G_2$, $y$ is a common neighbour of $w$ and $u$ in $\{l_{k-3},l_{k-2},l_{k-1},l_k\}$, there exists an element $x$ of $X$ distinct from $y$,
which implies that the tuple $M=M'+(y,u,x,v)=(m_1,...,w)+(y,u,x,v)$ is such that
$P_{M,G_2}$ exists. Observe that 
\newline 
$|V(P_{M,G_2})|=|V(Y)|+4=2*\lfloor (k-3)/2 \rfloor + 4= 2*\lfloor (k+1)/2 \rfloor$, 
\newline
$x,y \in V(P_{L,G_1})$ and $u,v \in Q$, which implies that
$P_{M,G_2}$ adheres to the given constraints, proving the lemma for this case.
\newline
\newline
Otherwise, if $|X|=1$ and $x$ is the element of $X$, assume without loss of generality that $x \in \{l_{k-3},l_{k-2}\}$.
By Observation \ref{obs:1}, for any $c_1 \in \{u,v,w\}$,
$c_1$ has a neighbour in $\{l_{k-1},l_k\}$.
\newline
Therefore, since $|\{u,v,w\}|=3$ and $|\{l_{k-1},l_k\}|=2$, by the Pigeonhole Principle, 
there exist two elements $c_1,c_2$ of $\{u,v,w\}$ with a common neighbour $y \in \{l_{k-1},l_k\}$. 
By assumption, as $|X|=1$, $w \in \{c_1,c_2\}$. Therefore, if $c_1=w$, 
as $c_2 \in \{u,v\}$ and is hence a neighbour of $x$, if $c_3$ is the non-$c_2$ element of $\{u,v\}$, as $x \neq y$, 
the tuple $M=M'+(y,c_2,x,c_3)=(m_1,...,w)+(y,c_2,x,c_3)$ is such that $P_{M,G_2}$ exists.
\newline
By the argument in the previous paragraph, $P_{M,G_2}$ adheres to the given constraints.
\newline
This proves the lemma. 
\end{proof}
\subsection{Construction of a copy of $T_n$ in $G_2$.}\label{subsec:2.2}

In the following proof, we construct a copy of $T_n$ in $G_2$  for the cases $z=0$,  $z \in (0, (n-1)/4]$ and
$z \in ((n-1)/4, (n-1)/2]$ in Lemma \ref{lem:3}, Lemma \ref{lem:4} and Lemma \ref{lem:5} respectively. 
\newline
\newline

\begin{lemma}\label{lem:3}
For $z=0$, $G_2$ contains a subgraph isomorphic to $T_n$.
\end{lemma}
\begin{proof}
If $z=0$, then $P=P_{(v_1,....,v_{n+1}),G_1}$ and
$|Q|=(3n+1)/2-(n+1)=(n-1)/2$. 
Since $n>5$ by assumption, it follows that $|Q| \geq 3$ and
hence has 3 distinct elements $x,y,y'$.
Consider the set $S=\{v_{n-1},v_{n}\}$. By Observation \ref{obs:1}, each element of $\{x,y,y'\}$ is adjacent to some element of $S$ in $G_2$, implying by
the Pigeonhole Principle that in $G_2$, some $v \in S$ is adjacent to two elements of $\{x,y,y'\}$, which we may assume without loss of generality to be 
$x$ and $y$. Since $n>5$ by assumption, $(n-1)/4 \geq 2$, implying that there exists a subset $Q_2$ of size $(n-1)/4$ of $Q$ that contains $x$ and $y$ as $(n)mod(4)=1$. 
\newline
\newline
Let $B_2$ be the subpath $P_{(v,....,v_{(n+1)/2+2},v_{(n+1)/2+1}), G_1}$ of $P$. 
By definition, since $v \in \{v_{n-1},v_n\}$, $|V(B_2)| \in \{(n-1)-(n+1)/2, n-(n+1)/2\}=\{(n-3)/2,(n-1)/2\}$, 
which implies that $\lfloor \frac{|V(B_2)|+1}{2} \rfloor = (n-1)/4 = |Q_2|$ as $(n)mod(4)=1$. 
\newline
\newline
Hence, since two elements of $Q_2$ are adjacent to $v$, by Lemma \ref{lem:2}, if $k'= (n-1)/4$, $P_{M_2,G_2}$ exists for some tuple  $M_2=(m_1,....,m_{2k'})=(m_1,...,m_{(n-1)/2})$ of elements of $V(K_r)$ such that $m_j \in V(B_2)$ iff $j$ is odd and $m_j \in Q_2$ otherwise. Let $S_2$ denote the $(n-1)/2$-path $P_{M_2,G_2}$ and let $t_2$ be the endpoint $m_{(n-1)/2}$ of $S_2$ in $Q_2$.
\newline
\newline
Let $Q_1$ be the $(n-1)/4$-set $Q \setminus Q_2$ and let $B_1$ be the subpath $P_{v_1,...,v_{(n-3)/2}, G_1}$ of $P$, which is vertex-disjoint from $B_2$. 
Observe that if some element of $Q_1$ is adjacent to $v_1$ in $G_1$, since $v_1$ is an endpoint of $P$, as $Q_1 \subseteq Q$ is disjoint from $V(P)$ by definition of $Q$,
$P$ is not a longest path of $G_1$, which is a contradiction by assumption. 
\newline
\newline
Therefore, each element of $Q_1$ is adjacent to $v_1$ in $G_2$. Hence, since $|Q_1|=(n-1)/4 \geq 2$ and $ \lfloor \frac{|V(B_1)|+1}{2} \rfloor = (n-1)/4 = |Q_1|$,
by Lemma \ref{lem:2}, if $k'= (n-1)/4$, $P_{M_1,G_2}$ exists for some tuple  $M_1=(m'_1,....,m'_{2k'})=(m'_1,...,m'_{(n-1)/2})$ of elements of $V(K_r)$ such that $m'_j \in V(B_1)$ iff $j$ is odd and $m_j \in Q_1$ otherwise. Let $S_1$ denote the $(n-1)/2$-path $P_{M_1,G_2}$ and let $t_1$ be the endpoint $m'_{(n-1)/2}$ of $S_1$ in $Q_1$.
\newline
\newline
\begin{observation}\label{obs:2}
 $S_1$ and $S_2$ are vertex-disjoint $(n-1)/2$-paths that both do not pass through any element of $\{v_{(n+1)/2},v_{n+1}\}$, which is a subset of $V(P) \setminus (V(B_1) \cup V(B_2))$
and is hence also disjoint from $Q$.
\end{observation}

\begin{observation}\label{obs:3}
If the midpoint $v_{(n+1)/2}$ of the $n$-path $P'=P_{(v_1,...,v_n),G_1}$ of $G_1$ is adjacent to some element $w  \in V(K_r) \setminus V(P')$ in 
$G_1$, the graph $G'=(V(P') \cup \{w\}, E(P') \cup \{v_{(n+1)/2}w\})$ is a copy of $T_n$ in $G_1$ with spine $P'$ and pendant $w$, which is a contradiction,
as by assumption, $T_n$ is not a subgraph of $G_1$. 
\newline
\newline
\end{observation}

Therefore, $v_{(n+1)/2}$ is adjacent to each element of $\{t_1,t_2,v_{n+1}\} \subseteq V(K_r) \setminus V(P')$ in $G_2$. 
Hence, if $M'_2$ is the reverse of the tuple of consecutive vertices $M_2$ of $S_2$, as $M'_2$ has initial element $t_2$ and $M_1$ has final element $t_1$,
by Observation \ref{obs:2}, 
$M_1+(v_{(n+1)/2})+M'_2$ is a tuple $M$ such that the $n$-path $P_{M,G_2}$ exists. 
Let $W$ be $P_{M,G_2}$. 
\newline
\newline
By Observation \ref{obs:2}, the neighbour $v_{n+1}$ of $v_{(n+1)/2}$ in $G_2$ does not lie on $W$.
Observe that since $|M_1|=|V(S_1)|$ and $|M'_2|=|V(S_2)|$ are both equal to $(n-1)/2$,
$W$ is of length $n$ and $v_{(n+1)/2}$ is the midpoint of $W$.
\newline
Therefore, the graph $G=(V(W) \cup \{v_{(n+1)}\}, E(W) \cup \{v_{(n+1)/2}v_{n+1}\})$ is 
a copy of $T_n$ in $G_2$ with spine $W$ and pendant $v_{n+1}$.
Hence, for $z=0$, $G_2$ contains a copy of $T_n$. 

To illustrate this proof, a pictorial representation of the copy $G$ of $T_n$ with pendant $v_{n+1}$ and $P$ where the edges of $P$ and $G$ are indicated in black and red respectively is displayed in Figure \ref{x} below.

\begin{figure}[H]    
\centering
\begin{tikzpicture}[scale=3]

  \draw (0.2,1.86) -- (3.11,1.88);
  \draw[red] (0.2,1.86) -- (0.4,2.2);
  \draw[draw=red] (0.4,2.2) -- (0.54,1.85);
  \draw[red] (0.55,1.86) -- (0.72,2.2);
  \draw[draw=red] (0.71,2.2) -- (0.85,1.87);
  \draw[draw=red] (0.84,1.88) -- (0.85,1.88);
  \draw[draw=red] (0.85,1.87) -- (1.07,2.2);
  \draw[red] (1.07,2.2) -- (1.26,1.87);
  \draw[draw=red] (1.28,1.88) -- (1.5,2.2);
  \draw[draw=red] (1.5,2.2) -- (1.69,1.88);
  \draw[draw=red] (1.7,1.87) -- (1.91,2.2);

  \draw[red] (1.28,1.88) .. controls (1.28,1.86) and (1.27,1.8) .. (1.29,1.75) .. controls (1.3,1.71) and (1.34,1.66) .. (1.36,1.63) .. controls (1.39,1.61) and (1.41,1.61) .. (1.44,1.6) .. controls (1.47,1.58) and (1.48,1.58) .. (1.51,1.57) .. controls (1.55,1.55) and (1.6,1.53) .. (1.63,1.53) .. controls (1.67,1.52) and (1.68,1.51) .. (1.71,1.51) .. controls (1.75,1.51) and (1.81,1.5) .. (1.84,1.5) .. controls (1.88,1.5) and (1.89,1.5) .. (1.92,1.5) .. controls (1.96,1.5) and (2.02,1.5) .. (2.07,1.5) .. controls (2.11,1.5) and (2.16,1.5) .. (2.2,1.5) .. controls (2.23,1.5) and (2.24,1.5) .. (2.29,1.5) .. controls (2.33,1.5) and (2.41,1.5) .. (2.46,1.51) .. controls (2.5,1.51) and (2.52,1.51) .. (2.56,1.52) .. controls (2.6,1.52) and (2.64,1.53) .. (2.68,1.54) .. controls (2.73,1.54) and (2.77,1.55) .. (2.81,1.56) .. controls (2.85,1.57) and (2.89,1.59) .. (2.93,1.61) .. controls (2.97,1.63) and (3.02,1.66) .. (3.05,1.69) .. controls (3.08,1.72) and (3.09,1.77) .. (3.11,1.8) .. controls (3.12,1.83) and (3.12,1.85) .. (3.12,1.86);
  \draw[fill=black] (3.12,1.88) circle (0.04cm);
  \draw[fill=black] (1.28,1.87) circle (0.04cm);
  \draw[red] (2.14,1.87) -- (2.27,2.2);
  \draw[red] (1.91,2.2) -- (2.15,1.87);
  \draw[red] (2.27,2.2) -- (2.53,1.87);
  \node[] at (1.13,1.77) {$v_{\frac{(n+1)}{2}}$};

  \node[] at (3.31,1.8) {$v_{n+1}$};
\end{tikzpicture}
\caption{}
\label{x}
\end{figure}

\end{proof}

\begin{lemma}\label{lem:4}
For $z \in (0,(n-1)/4]$, $G_2$ contains a subgraph isomorphic to $T_n$.
\end{lemma}
\begin{proof}
Recall that $|Q|=(3n+1)/2-|V(P)|=(3n+1)/2-(n+1)-z=(n-1)/2-z$.
By assumption, $|Q| \geq (n-1)/4$. Let $Q_1$ be a subset of $Q$ of size $(n-1)/4$, which is an integer as $(n)mod(4)=1$. 
Let $B_1$ be the subpath $P_{(v_1,...,v_{(n-3)/2}),G_1}$ of $P$. 
\newline
Recall that each element of $Q_1$ is adjacent to $v_1$ in $G_2$ by the maximality of $P$ in $G_1$, $|Q_1|=\lfloor (|V(P)|+1)/2 \rfloor$,
and $|Q_1| \geq 2$ since $n>5$. 
\newline
Therefore, by Lemma \ref{lem:2}, if $k'= (n-1)/4$, $P_{M_1,G_2}$ exists for some tuple  $M_1=(m'_1,....,m'_{2k'})=(m'_1,...,m'_{(n-1)/2})$ of elements of $V(K_r)$ such that $m'_j \in V(B_1)$ iff $j$ is odd and $m_j \in Q_1$ otherwise. Let $S_1$ denote the $(n-1)/2$-path $P_{M_1,G_2}$ and let $t_1$ be the endpoint $m'_{(n-1)/2}$ of $S_1$ in $Q_1$.
\newline
\newline
Recall that by Lemma \ref{lem:1}, if $z \neq 0$ and  $M_2=(v_{n+2},v_{(n+2)-(n+1)/2})+(v_{n+3},v_{(n+3)-(n+1)/2})+,...,+(v_{n+z},v_{n+z-(n+1)/2})+(v_{n+z+1},v_{n+z+1-(n+1)/2})$,
the $2z$-path $P_{M_2,G_2}$ exists. Let $S_2$ denote the path $P_{M_2,G_2}$. 
\newline
\newline
Suppose that $z=(n-1)/4$. 
\newline
By assumption, $|V(S_2)|=(n-1)/2$.
By Observation \ref{obs:3} in the proof of Lemma \ref{lem:3}, in $G_2$, $v_{(n+1)/2}$ is adjacent to the elements $t_1$ and  $v_{n+2}$ of $V(K_r) \setminus V(P')$, which are respectively the final and initial elements of $M_1$ and $M_2$.
Therefore, since $S_1$ and $S_2$ are vertex-disjoint paths that do not pass through any element of $\{v_{(n+1)/2},v_{n+1}\}$, 
$P_{M_1+(v_{(n+1)/2})+M_2,G_2}$ is a path $W$ of length $n$ in $G_2$ with midpoint $v_{(n+1)/2}$, which is adjacent to $v_{n+1} \in V(K_r) \setminus (V(S_1) \cup V(S_2))$.
\newline
Therefore, the graph $G=(V(W) \cup \{v_{(n+1)}\}, E(W) \cup \{v_{(n+1)/2}v_{n+1}\})$ is 
a copy of $T_n$ in $G_2$ with spine $W$ and pendant $v_{n+1}$.
\newline
Hence, for $z=(n-1)/4$, $G_2$ contains a copy of $T_n$, and we assume henceforth that $z<(n-1)/4$, which implies that $|Q|>(n-1)/4$.
\newline
\newline
Let $Q_3$ be $Q \setminus Q_1$. Observe that since $|Q|=(3n+1)/2-|V(P)|=(3n+1)/2-(n+1+z)=(n-1)/2-z$, $|Q_3|=(n-1)/2-z-(n-1)/4=(n-1)/4-z > 0$.
\newline
We initially prove the following claim.
\newline
\newline
\begin{claim} \label{clm:2}
The lemma holds for $|Q_3| \geq 3$.
\end{claim}
\begin{proof}
In order to obtain a copy of $T_n$ in $G_2$, we use Lemma \ref{lem:2} to construct a path $S_3$ of length $\frac{n-1}{2}-2z$ in $G_2$ with tuple of consecutive vertices $M_{3}$ such that $V(S_3)$ is disjoint from the sets $V(P_{M_1,G_2})$, $V(P_{M_2,G_2})$ and $\{v_{(n+1)/2}, v_{n+1}\}$, and we then prove that there exists a copy of $T_n$ in $G_2$ with spine $P_{M_{1}+(v_{(n+1)/2})+M_{2}+M_{3}, G_2}$ and pendant $v_{n+1}$ attached to $v_{(n+1)/2}$. 
\newline
\newline
Consider the set $S=\{v_{n-1},v_n\}$. Recall that since $P$ is a longest path in $G_1$, each element 
of $Q_3$ (which is a subset of $Q=V(G_1) \setminus V(P)$) is adjacent to some element of $S$ in $G_2$. Therefore, as $|Q_3| \geq 3$, by the Pigeonhole Principle, there exists $v_t$ adjacent to at least two elements of $Q_3$ for some $t \in \{n-1,n\}$ in $G_2$. 
Observe that the length of the subpath $P'_3$ of $P$ from $v_{(n+1)/2+z+1}$ to $v_t$ is at least $(n-1)-((n+1)/2+z)=(n-1)/2-z-1 \geq (n-1)/2-2z$ as 
$z>0$ by assumption. 
Since $(n-1)/4 > z$ by assumption, $(n-1)/2-2z>0$. 
\newline
\newline
Therefore, $t>(n+1)/2+z+1$ and there exists a subpath $B_3=P_{(v_t,...,v_{t-((n-1)/2-2z+1)}),G_1}$ of $P'_3$ of length $(n-1)/2-2z \geq 1$.
\newline
\newline
Observe that $B_3$ is a subpath of $P$ such that $|Q_3|=(n-1)/4-z=\lfloor (|V(B_3)|+1)/2 \rfloor$ and such that the endpoint $v_t$ of $B_3$ is adjacent to at least two elements of $Q_3$.
\newline
\newline
Therefore, by Lemma \ref{lem:2}, if $k'= (n-1)/4-z$, $P_{M'_3,G_2}$ exists for some tuple  $M'_3=(m_1,....,m_{2k'})=(m_1,...,m_{(n-1)/2-2z})$ of elements of $V(K_r)$ such that $m_j \in V(B_3)$ iff $j$ is odd and $m_j \in Q_3$ otherwise. Let $S_3$ denote the $((n-1)/2-2z)$-path $P_{M'_3,G_2}$ and let $t_3$ be the endpoint $m_{(n-1)/2-2z}$ of $S_3$ in $Q_3$.
Let $M_3$ be the reverse of $M'_3$. By definition, $M_3$ has initial vertex $t_3$ and $P_{M'_3,G_2}$ is equivalent to $P_{M_3,G_2}$.
\newline
\newline
Recall that $V(S_3)=((V(S_3) \cap V(P)) \cup Q_3) \subseteq \{v_{{(n+1)/2}+z+1},....,v_n\} \cup Q_3$ and that $V(S_2) \subseteq \{v_{(n+1)/2+1},....,v_{(n+1)/2+z}\} \cup \{v_{(n+2)},...,v_{(n+1)+z}\}$. Therefore, since $Q_3 \subseteq Q$ is disjoint from $P=P_{(v_1,...,v_{(n+1)+z}),G_1}$, it follows that $S_3$ and $S_2$ are vertex-disjoint.
Similarly, since $V(S_1) \subseteq \{v_1,...,v_{(n-1)/2}\} \cup Q_1$ and $Q_1=Q \setminus Q_3$, $S_3$ and $S_1$ are vertex-disjoint.
\newline
\newline
Therefore, for the disjoint paths $S_1=P_{M_1,G_2}, S_2=P_{M_2,G_2}, S_3=P_{M_3,G_2}$, if in $G_2$, the final vertex $t_1$ of $M_1$ and the initial vertex $v_{n+2}$ of $M_2$ are both  adjacent to $v_{(n+1)/2}$
and the final vertex $v_{(n+1)/2+z}$ of $M_2$ is adjacent to the initial vertex $t_3$ of $M_3$, the tuple $X=M_1+(v_{(n+1)/2})+M_2+M_3$ is such that 
$P_{X,G_2}$ is a well-defined path in $G_2$. 
\newline
\newline
If $P_{X,G_2}$ exists, since $|M_1|=|V(S_1)|=(n-1)/2, |M_2|=|V(S_2)|=2z$ and $|M_3|=|V(S_3)|=(n-1)/2-2z$, it follows that $P_{X,G_2}$ is a path of length n with midpoint $v_{(n+1)/2}$. 
\newline
Recall that $v_{n+1}$ does not lie on any path in $\{S_1,S_2,S_3\}$ and is adjacent to $v_{(n+1)/2}$ in $G_2$ by Observation \ref{obs:3} in the proof of Lemma \ref{lem:3}. 
Therefore, the graph $G=(V(P_{X,G_2}) \cup \{v_{n+1}\},E(P_{X,G_2}) \cup \{v_{(n+1)/2}v_{n+1}\})$ is a copy of $T_n$ in $G_2$ with spine $P_{X,G_2}$ and pendant $v_{n+1}$, proving
the claim given that $P_{X,G_2}$ exists.
\newline
\newline
Observe that $t_1 \in Q, v_{n+2} \in V(P) \setminus \{v_1,...,v_n\}$
and that $Q$, $V(P) \setminus \{v_1,...,v_n\}$ are both disjoint from $V(P')=\{v_1,...,v_n\}$.
Therefore, by Observation \ref{obs:3} in the proof of Lemma \ref{lem:3}, $v_{(n+1)/2}$ is adjacent to both $t_1$ and $v_{n+2}$ in $G_2$. 
Hence, to prove the existence of $P_{X,G_2}$, we need only prove that $e=v_{(n+1)/2+z}t_3 \in E(G_2)$.
\newline
Suppose that $e \in E(G_1)$. Observe that $((n+1)/2+z)+(n-1)/2=n+z<(n+1)+z=|V(P)|$ and
$((n+1)/2+z)-(n-1)/2=z+1>0$. 
Therefore, $T'=P_{(v_{((n+1)/2+z)-(n-1)/2},...,v_{((n+1)/2+z)+(n-1)/2}),G_1}$ is a well-defined subpath of P of length $n$ with midpoint
$v_{(n+1)/2+z}$.
\newline
Recall that by definition, $t_3 \in Q$ and hence is not an element of $V(P)=V(K_r) \setminus Q$, which contains $V(T')$.
Therefore, since the midpoint of the $n$-path $T'$ in $G_1$ is adjacent to $t_3$ in $G_1$ by assumption,
there exists a copy of $T_n$ in $G_1$ with spine $T'$ and pendant $t_3$, which is a contradiction as
$G_1$ contains no copy of $T_n$ by assumption. 
\newline
\newline
Therefore, by definition of $G_1$ and $G_2$, $e \in E(G_2)$, proving that $P_{X,G_2}$ exists.
This proves the claim.
For clarification, $P$ and $G$ (with pendant $v_{n+1}$) are represented in Figure \ref{y} in black and red respectively. 
\begin{figure}[H]  
\centering
\begin{tikzpicture}[scale=4]

  \draw (-0.4,1.47) -- (2.71,1.47);
  \draw[red] (-0.4,1.47) -- (-0.31,1.8);
  \draw[draw=red] (-0.31,1.8) -- (-0.22,1.46);
  \draw[red] (-0.22,1.47) -- (-0.13,1.8);
  \draw[red] (-0.13,1.8) -- (-0.04,1.48);
  \draw[red] (-0.04,1.47) -- (0.05,1.8);

  \draw[red] (0.05,1.8) -- (0.15,1.48);

  \draw[red] (0.15,1.47) -- (0.26,1.8);

  \draw[red] (0.25,1.8) -- (0.35,1.48);

  \draw[red] (0.35,1.47) -- (0.46,1.8);

  \draw[red] (0.45,1.8) -- (0.55,1.48);

  \draw[red] (0.55,1.47) -- (0.66,1.8);

  \draw[red] (0.65,1.8) -- (0.75,1.48);

  \draw[red] (0.75,1.47) .. controls (0.76,1.45) and (0.78,1.36) .. (0.82,1.33) .. controls (0.87,1.3) and (0.94,1.28) .. (1,1.27) .. controls (1.05,1.26) and (1.11,1.27) .. (1.16,1.27) .. controls (1.22,1.27) and (1.27,1.27) .. (1.33,1.27) .. controls (1.39,1.27) and (1.45,1.27) .. (1.51,1.27) .. controls (1.58,1.27) and (1.64,1.27) .. (1.7,1.28) .. controls (1.76,1.28) and (1.83,1.29) .. (1.89,1.29) .. controls (1.95,1.29) and (2.01,1.28) .. (2.06,1.29) .. controls (2.11,1.3) and (2.16,1.33) .. (2.21,1.36) .. controls (2.25,1.39) and (2.3,1.45) .. (2.32,1.47);

  \draw[red] (0.45,1.8) -- (0.55,1.48);

  \draw[red] (0.89,1.48) .. controls (1.51,0.87) and (2,0.9) .. (2.31,1.45);
  \draw[red] (0.9,1.48) .. controls (1.8,0.14) and (2.2,0.89) .. (2.47,1.48);
  \draw[red] (2.48,1.48) .. controls (2.4,0.72) and (1.65,-0.2) .. (1.1,1.48);
  \draw[red] (1.11,1.48) .. controls (1.4,-0.2) and (2.4,0.13) .. (2.71,1.48);
  \draw[red] (2.71,1.47) .. controls (2.87,-0.28) and (1,-0.27) .. (1.27,1.47);
  \draw[fill=black] (2.31,1.47) circle (0.03cm);
  \draw[red] (1.27,1.47) -- (1.32,1.8);
  \draw[red] (1.33,1.8) -- (1.4,1.47);
  \draw[red] (1.45,1.81) -- (1.52,1.48);

  \draw[red] (1.4,1.47) -- (1.45,1.8);

  \draw[red] (1.52,1.47) -- (1.6,1.8);

  \draw[red] (1.6,1.8) -- (1.67,1.47);
  \draw[fill=black] (1.27,1.47) circle (0.03cm);
  \draw[red] (0.75,1.47) .. controls (1.08,2.35) and (1.6,2.6) .. (2.11,1.47);
  \draw[fill=black] (2.11,1.47) circle (0.03cm);
  \draw[fill=black] (0.75,1.47) circle (0.03cm);
  \draw[fill=black] (2.71,1.47) circle (0.03cm);
  \node[draw=none, node font=\small] at (0.65,1.36) {$v_{\frac{n+1}{2}}$};
  \node[draw=none, node font=\small] at (1.43,1.36) {$v_{\frac{n+1}{2}+z}$};

  \node[draw=none, node font=\small] at (1.96,1.56) {$v_{n+1 
} $};

  \node[draw=none, node font=\small] at (2.3,1.56) {$v_{n+2}
$};

  \node[draw=none, node font=\small] at (2.9,1.55) {$v_{n+1+z}
  $};
  \draw[fill=black] (1.32,1.8) circle (0.03cm);

  \node[draw=none, node font=\small] at (1.32,1.9) {$t_3 
 $};
\end{tikzpicture}
\caption{}
\label{y}
\end{figure}

\end{proof}

By Claim \ref{clm:2}, to prove the lemma, we need only prove the existence of a copy of $T_n$ in $G_2$
for $|Q_3|=2$ and $|Q_3|=1$. 
Consider the $((n-1)/2-2z)$-subpath $B_3=P_{(v_t,...,v_{t-(n-1)/2-2z+1}),G_1}$ of $P$ for $t=(n-1)$ as defined in Claim \ref{clm:2}. 
Observe that $|V(B_3)|=2*((n-1)/4-z)=2*|Q_3|$.
Therefore, if $|Q_3|=2$, $B_3=P_{(v_t,v_{t-1},v_{t-2},v_{t-3}),G_1}$ and the elements of $Q_3$ are vertices $k_1$,$k_2$.
\newline
\newline
Recall that by the proof of Claim \ref{clm:1}, since in $G_1$, $P$ is a longest path, it follows that in $G_2$, $k_1$ and $k_2$ have a common neighbour 
$v_l \in \{v_t,v_{t-1},v_{t-2},v_{t-3}\}$ and each $k_j$ where $j \in \{1,2\}$ has a neighbour in both $\{v_t,v_{t-1}\}$
and $\{v_{t-2},v_{t-3}\}$. Therefore, if we assume without loss of generality that $v_l \notin \{v_{t-2},v_{t-3}\}$, if $v_m  \in \{v_{t-2},v_{t-3}\}$ is a neighbour of $k_2$ in $G_2$, 
 $P_{(k_1,v_l,k_2,v_m),G_2}$ is a well-defined 4-path which has an endpoint $k_1$ in $Q_3$ and alternates between $Q_3$ and $V(B_3)$. 
Let $T'_3$ be this path and let $M_3$ be $(k_1,v_l,k_2,v_m)$.
\newline
\newline
Observe that $|V(T'_3)|=4=2*|Q_3|=(n-1)/2-2z$. 
Therefore, since $T'_3=P_{(k_1,v_l,k_2,v_m),G_2}$ (where $k_1,k_2 \in Q_3$ and $v_1,v_m \in V(B_3)$) and the path $S_3$ in Claim \ref{clm:2} also has length $(n-1)/2-2z$,
we may apply the argument in Claim \ref{clm:2} (with $k_1$ instead of $t_3$) to prove that $P_{M_1+(v_{(n+1)/2})+M_2+M_3,G_2}$ is a well defined path with midpoint $v_{(n+1)/2}$ and that $G_2$
contains a copy of $T_n$ with spine $P_{M_1+(v_{(n+1)/2})+M_2+M_3,G_2}$ and pendant $v_{n+1}$, proving the lemma for $|Q_3|=2$. 
\newline
\newline
Similarly, for $|Q_3|=1$, $B_3=P_{(v_t,v_{t-1}),G_1}$ and the element $k_1$ of $Q$ has a neighbour $v_l$ in $B_3$, implying the existence of the path 
$T'_3=P_{(k_1,v_l),G_2}$, which alternates between $Q_3$ and $V(B_3)$ and has length equal to $2=2*|Q_3|=(n-1)/2-2z$.
Therefore, by the argument in the previous paragraph, $G_2$ contains a copy of $T_n$ with spine $P_{M_1+(v_{(n+1)/2})+M_2+M_3,G_2}$ where $M_3=(k_1,v_l)$ and pendant $v_{n+1}$, proving the lemma for $|Q_3|=1$. 
\newline
This proves the lemma. 
\end{proof}

\begin{lemma}\label{lem:5}
For $z>(n-1)/4$, $G_2$ contains a subgraph isomorphic to $T_n$.
\end{lemma}
\begin{proof}
Recall that as $z \neq 0$, by Lemma \ref{lem:1}, if $V$ is the tuple 
$(v_{n+2},v_{(n+2)-(n+1)/2})+(v_{n+3},v_{(n+3)-(n+1)/2})+,...,+$
$+(v_{n+z},v_{n+z-(n+1)/2})+(v_{n+z+1},v_{n+z+1-(n+1)/2})$,
the $2z$-path $P_{V,G_2}$ exists.
\newline
In this lemma, let $M_1$ denote the tuple $V$ and let
$S_1$ denote $P_{V,G_2}$. 
\newline
\newline
Recall that $z \leq (n-1)/2$. 
\newline
Suppose that $z=(n-1)/2$.
Then, $|V(M_1)|=(n-1)$ and the final vertex of $M_1$ is $v_n$.
Observe that since $z=(n-1)/2$, $S=P_{(v_{n-z},....,v_{n+z}),G_1}$ is a subpath of $P=P_{(v_1,....,v_{(n+1)+z}),G_1}$ of length
$n$ with midpoint $v_n$. Since $n>1$ by assumption, $v_1$ does not lie on $S$. 
\newline
Therefore, if $v_{n}v_{1} \in E(G_1)$,
$G_1$ contains a copy of $T_n$ with spine $S$ and pendant $v_1$, which is a contradiction.
Therefore, $v_{n}v_{1} \in E(G_2)$.
\newline
\newline
Hence, since $v_1$ does not lie on $S_1$ by definition, if $M'_1=M_1+(v_1)$,
the path $T'_1=P_{M'_1,G_2}$ is a well-defined $n$-path of $G_2$. 
By definition of $T'_1$, as $(n)mod(4)=1$, the midpoint of $T'_1$ is $v_{(n+2)+(n-1)/4}$.
Since the midpoint $v_{(n+2)+(n-1)/4}$ of the $n$-path $T'_1$ is not contained in $P'=P_{(v_1,...,v_n),G_1}$, by Observation \ref{obs:3}, in $G_2$,
$v_{(n+2)+(n-1)/4}$ is adjacent to $v_{(n+1)/2}$, which does not lie on $T'_1$, which implies that $G=(V(T'_1) \cup \{v_{(n+1)/2}\}, E(T'_1) \cup \{v_{(n+2)+(n-1)/4}v_{(n+1)/2}\})$ is
a copy of $T_n$ in $G_2$ with spine $T'_1$ and pendant $v_{(n+1)/2}$, proving the lemma for $z=(n-1)/2$.
\newline
\newline
Therefore, we may assume henceforth that $z \leq (n-1)/2-1$. 
\newline
\newline
Since $z \leq ((n-1)/2)-1$ by assumption, $(n-2)-2z \geq 2-1=1$, which implies that
the path $B_2=P_{(v_1,...,v_{(n-2)-2z}),G_1}$ is a well-defined subpath of $P$. 
\newline
If $(n-1)-2z>(n-3)/2$, $(n+1)/2>2z$, implying that $(n+1)/4>z$, implying that $(n-1)/4 \geq z$ as $(n)mod(4)=1$, which is a contradiction
as $z>(n-1)/4$ by assumption.
\newline
 Therefore, $(n-1)-2z \leq (n-3)/2$, implying that $(n-2)-2z<(n-3)/2$, implying that 
$B_2$ is a subpath of $P_{(v_1,....,v_{(n-5)/2}),G_1}$, which is well-defined as $n>5$. 
\newline
\newline
Recall that $|Q|=(n-1)/2-z=\lfloor ((n-1)-2z)/2 \rfloor = \lfloor (|V(B_2)|+1)/2 \rfloor$. 
Therefore, since each element of $Q$ is adjacent to the endpoint $v_1$ of $B_2$, by Lemma \ref{lem:2}, if $k'=\lfloor ((n-2)-2z+1)/2 \rfloor$, $P_{M_2,G_2}$ exists for some tuple  $M_2=(m_1,....,m_{2k'})$ of elements of $V(K_r)$ such that $m_j \in V(B_2)$ iff $j$ is odd and $m_j \in Q$ otherwise
and $m_1 \in \{v_1,v_2,v_3\}$.
\newline
\newline
Recall that since $P$ is a longest path of $G_1$, $m_{2k'}$ is adjacent in $G_2$ to at least one element $v$ of the set $\{v_{(n-3)/2},v_{(n-1)/2}\} \subseteq V(K_r) \setminus V(P_{M_2,G_2})$, which consists of two consecutive vertices of $P$. 
Therefore, since $v$ does not lie on $P_{M_2,G_2}$, $S_2=P_{M_2+(v),G_2}$ is a well-defined $(n-2z)$-path in $G_2$. 
\newline
\newline
Consider the final vertex $v_{n+z+1-(n+1)/2}$ of $M_1$ and let $k$ be $n+z+1-(n+1)/2$. 
\newline
Since $((n+z+1)-(n+1)/2)+(n-1)/2=n+z<(n+1)+z=|V(P)|$ and $((n+z+1)-(n+1)/2)-(n-1)/2=z+1>(n-1)/4+1 \geq (9-1)/4+1=3$ as $n>5$ and $(n)mod(4)=1$, $S=P_{(v_{k-(n-1)/2},....,v_{k+(n-1)/2}),G_1}$ is a well-defined subpath of $P$ of length $n$ with midpoint $v_k$ that does not pass through any vertex in $\{v_1,v_2,v_3\}$.
\newline
\newline
Therefore, for each vertex $w \in \{v_1,v_2,v_3\}$, if $v_k$ is adjacent to $w$, $G_1$ contains a copy of $T_n$ with spine $S$ and pendant $w$, which is a contradiction as by assumption, $G_1$ contains no copy of $T_n$. Therefore, in $G_2$, the final vertex $v_k$ of the tuple $M_1$ of consecutive vertices of $S_1$ is adjacent to the initial vertex $m_1$ of the tuple of consecutive vertices $M_{2}+(v)$ of $S_2$, as $m_{1} \in \{v_1,v_2,v_3\}$.
Recall that $S_1$ and $S_2$ are disjoint.
Hence, for the tuple $M=M_1+M_2+(v)$, $P_{M,G_2}$ exists. By definition, $|M|=|M_1|+|M_2+(v)|=2z+(n-2z)=n$, which implies that $P_{M,G_2}$ is an $n$-path of $G_2$.
Since $v_{(n+1)/2}$ does not lie on $S_1$ or $S_2$, $v_{(n+1)/2}$ does not lie on $P_{M,G_2}$.
\newline
\newline
Recall that since $z \geq (n-1)/4+1$, $|M_1| \geq (n+1)/2$. Therefore the midpoint of $P_{M,G_2}$ is the entry $M_{1}((n+1)/2)$ of the tuple $M_1$, which is $v_{(n+2)+(n-1)/4}$.
Recall that by Observation \ref{obs:3}, $v_{(n+2)+(n-1)/4}v_{(n+1)/2} \in E(G_2)$. 
Hence, the graph $G=(V(P_{M,G_2}) \cup \{v_{(n+1)/2}\}$, $E(P_{M,G_2}) \cup \{v_{(n+2)+(n-1)/4}v_{(n+1)/2}\})$ is a copy of $T_n$ in $G_2$ with spine $P_{M,G_2}$ and pendant $v_{(n+1)/2}$, proving the lemma for $z<(n-1)/2$.
\newline
This proves the lemma. 
For clarification, $P$ and the copy $G$ of $T_n$ with pendant $v_{\frac{n+1}{2}}$ are represented pictorially in Figure \ref{z}.
\begin{figure}[H]   
\centering
\begin{tikzpicture}[scale=4]

  \draw (-0.59,2) -- (2.72,2);
  \draw[red] (2,2) .. controls (1.38,1.82) and (0.76,1.82) .. (0.14,2);
  \draw[red] (0.16,2) .. controls (1.08,1.28) and (2.22,2) .. (2.1,2);
  \draw[red] (2.12,2) .. controls (2.2,2) and (1.06,0.88) .. (0.4,2);
  \draw[red] (0.4,2) .. controls (0.6,1.12) and (1.67,1.09) .. (2.3,2);
  \draw[red] (2.3,2) .. controls (1.9,0.8) and (0.94,1.2) .. (0.6,2);
  \draw[red] (0.6,2) .. controls (1.74,1.09) and (2.14,1.11) .. (2.51,2);
  \draw[red] (2.51,2) .. controls (2.47,0.73) and (1.4,0.8) .. (0.86,2);
  \draw[red] (0.87,2) .. controls (1.06,0.4) and (2.7,0.7) .. (2.72,2);

  \draw[red] (2.72,2) .. controls (2.67,1.71) and (1.52,1.07) .. (1.11,2);
  \draw[red] (1.12,2) .. controls (0.4,1.14) and (-0.46,1.8) .. (-0.47,2);
  \draw[red] (2.51,2) .. controls (1.68,2.72) and (0.84,2.72) .. (0,2);
  \draw[red] (-0.47,2) -- (-0.3,2.52);
  \draw[draw=red] (-0.3,2.52) -- (-0.17,2);
  \draw[fill=black] (2,2) circle (0.03cm);
  \draw[fill=black] (2.51,2) circle (0.03cm);
  \draw[fill=black] (0,2) circle (0.03cm);
  \node[draw=none, node font=\small] at (0,1.85) {$v_{\frac{n+1}{2}}$};

  \node[draw=none, node font=\small] at (1.99,2.11) {$v_{n+2}
$};

  \node[draw=none, node font=\small] at (2.63,2.1) {$v_{n+2+\frac{n-1}{4}}$};
	
\end{tikzpicture}
\caption{}
\label{z}
\end{figure}

\end{proof}

\section{$R(T_n,T_n) \geq \frac{3n+1}{2}$ for all $n$ and $R(T_n,T_n)=\frac{3n+1}{2}$ if $n \leq 5$.} \label{sec:3}

In order to prove that $R(T_n,T_n) \geq \frac{3n+1}{2}$, by definition, for all $n$, we need only exhibit a 2-edge-colouring of $K_{\frac{3n+1}{2}-1}=K_{\frac{3n-1}{2}}$ such that
the graph induced by each colour class contains no copy of $T_n$. 
\newline
\newline
Let $\{u_1,....,u_{\frac{3n-1}{2}}\}$ be the vertex set $V(K_{\frac{3n-1}{2}})$ of $K_{\frac{3n-1}{2}}$, let $U_1$ be $\{u_1,....,u_n\}$,
let $U_2$ be $\{u_{n+1},...,u_{\frac{3n-1}{2}}\}$ and let $C$ be the set of edges of $K_{\frac{3n-1}{2}}$ with one endpoint in $U_j$ for each $j \in \{1,2\}$. 
\newline
Let $f$ be a 2-edge-colouring of $K_{\frac{3n-1}{2}}$ with colour classes $C_1=E(U_1) \cup E(U_2)$, $C_2=C$, and for each $x \in \{1,2\}$, let $G_x$ be $(V(K_{\frac{3n-1}{2}}),C_x)$.
\newline 
Since the components of $G_1$ have respective sizes $n$ and $(n-1)/2$ and $T_n$ has $n+1$ vertices, $G_1$ contains no copy of $T_n$.
Observe that $G_2$ is isomorphic to the biclique $K_{n,\frac{n-1}{2}}$ with partite sets $U_1$ and $U_2$.
\newline
Suppose that $G_2$ contains a copy of $T_n$ with spine $S=P_{(t_1,...,t_n), G_2}$ and pendant $t$. 
If $t_1 \in U_j$ for some fixed $j \in \{1,2\}$, as $G_2$ is a biclique with partite sets $U_1,U_2$, we may straightforwardly prove
by induction on $k$ that $t_k \in U_j$ iff $k$ is odd. 
\newline
\newline
Therefore, if $t_1 \in U_1$, since $(n+1)/2$ is odd as $(n)mod(4)=1$, the midpoint $t_{(n+1)/2}$ of $S$ is an element of $U_1$. 
\newline
Therefore, since $t$ does not lie on $S$ and is adjacent to $t_{(n+1)/2}$, $t \in U_2 \setminus V(S)$, which implies that $U_2 \setminus V(S)$ is non-empty.
However, since $n$ is odd, $U_2 \cap V(S)=\{t_2,t_4,...,t_{n-1}\}$, which has order $(n-1)/2=|U_2|$. Therefore, $|U_2 \setminus V(S)|=|U_2 \setminus (U_2 \cap V(S))|=(n-1)/2-(n-1)/2=0$,
implying that $U_2 \setminus S$ is empty, which is a contradiction.
\newline
\newline
Therefore, $t_1 \in U_2$, which implies that $U_2=\{t_1,t_3,....,t_n\}$, which has order $(n+1)/2$ since $n$ is odd, which is impossible since
$U_2=\{u_{n+1},....,u_{\frac{3n-1}{2}}\}$ and is hence a set of order $(n-1)/2$.
\newline
\newline
Hence, it follows that the graph induced by each colour class of $f$ contains no copy of $T_n$, which proves that 
$R(T_n,T_n) \geq \frac{3n+1}{2}$.
\newline
\newline
If $n \leq 5$, since $(n)mod(4)=1$, $n \in \{1,5\}$.
If $n=1$, then $K_{(3n+1)/2}=K_{2}$. Since the sole edge $e$ of $K_{2}$ induces
a copy of $T_1=T_n$ and is an element of exactly one colour class under any 2-edge-colouring of $K_{(3n+1)/2}$,
it follows that $R(T_1,T_1)=2$.
\newline
\newline
If $n=5$, $R(P_6,P_6)=(3*6-2)/2=8=(3*5+1)/2=(3n+1)/2$ (\cite{PathRamseyNumber}). 
Therefore, if $f$ is a 2-edge-colouring of $K_{(3n+1)/2}$ with colour classes
that respectively induce subgraphs $G_1,G_2$,
we may assume without loss of generality that $G_1$ contains a longest path $P=P_{(v_1,....,v_k),G_1}$ with $k \geq 6$ 
and that $G_1$ contains no copy of $T_5$. 
\newline
\newline
If $k=6$, let $u,w$ be the two elements of $K_{(3n+1)/2} \setminus V(P)$. 
Since $P$ is a longest path and $G_1$ contains no copy of $T_5$, $\{v_{3}v_{6},v_{6}w,wv_{1},v_{1}u\} \subseteq E(G_2)$, 
implying that $P_{(v_3,v_6,w,v_{1},u),G_2}$ exists. If $wv_4 \in E(G_1)$,
$G_1$ has a copy of $T_5$ with spine $P_{(v_2,v_3,v_4,v_5,v_6),G_1}$ and pendant $w$, which is a contradiction. 
Therefore, $wv_4 \in E(G_2)$, implying that $G_2$ contains a copy of $T_5$ with spine $P_{(v_3,v_6,w,v_{1},u),G_2}$ and pendant $v_4$.
\newline
\newline
If $k=7$, if $u$ is the element of $K_{(3n+1)/2} \setminus V(P)$, since $P$ is a longest path and $G_1$ contains no copy of $T_5$, 
$\{v_{6}v_{3},v_{3}v_{7},v_{7}v_{4},v_{4}v_{1}\} \subseteq E(G_2)$. Therefore, since $v_{7}u \in E(G_2)$ as $P$ is a longest path of $G_1$ with endpoint $v_{7}$,
$G_2$ contains a copy of $T_5$ with spine $P_{(v_6,v_3,v_7,v_4,v_1),G_2}$ and pendant $u$. 
\newline
\newline
If $k=8$, since $G_1$ contains no copy of $T_5$, $\{v_{7}v_{4},v_{4}v_{8},v_{8}v_{5},v_{5}v_{1},v_{8}v_{3}\} \subseteq E(G_2)$, implying that $G_2$ contains a copy of $T_5$ with spine
$P_{(v_7,v_4,v_8,v_5,v_1),G_2}$ and pendant $v_{3}$.
\newline
\newline
Therefore, we have proved that $R(T_5,T_5)=(3*5+1)/2$, proving the statement for $n=5$. 

\section*{Acknowledgements}
I thank Professor Vadim Lozin for his endorsement of this paper for submission to arXiv and for bringing paper \cite{Asymptote} to my attention.


\begin{thebibliography}{1}

\bibitem{Caterpillar1}
G.~Chen and S.~Liu.
\newblock On the {R}amsey numbers of even-linked double stars.
\newblock {\em Australasian Journal of Combinatorics}, 94:385--402, 2026.

\bibitem{Broom1}
P.~Erd{\H o}s, R.~Faudree, C.~Rousseau and R.~Schelp.
\newblock Ramsey numbers for brooms.
\newblock {\em Congressus Numerantium}, 35:283--293, 1982.

\bibitem{PathRamseyNumber}
L.~Gerencs{\'e}r and A.~Gy{\'a}rf{\'a}s.
\newblock On {R}amsey-type problems.
\newblock {\em Ann. Univ. Sci. Budapest. E{\"o}tv{\"o}s Sect. Math},
  10:167--170, 1967.




\bibitem{Broom2}
Y.~Li and P.~Yu.
\newblock All {R}amsey {N}umbers for {B}rooms in {G}raphs.
\newblock {\em Electronic Journal of Combinatorics}, 2016.

\bibitem{Notes}
V.~Lozin.
\newblock Graph {T}heory {N}otes.
\newblock Lecture notes, Institute of {M}athematics, {U}niversity of {W}arwick,
  {UK}, 2018.

\bibitem{Asymptote}
R.~Montgomery, M.~Pavez-Sign{\'e} and J.~Yan.
\newblock Ramsey numbers of trees.
\newblock arXiv preprint arXiv:2509.07934, 2025.

\bibitem{Caterpillar2}
Z.H. Sun.
\newblock Ramsey numbers for trees {II}.
\newblock {\em Czechoslovak Mathematical Journal}, 71:351--372, 2021.

\bibitem{Book}
D.B. West.
\newblock {\em Introduction to {G}raph {T}heory}.
\newblock Prentice {H}all, 1996.



\end{thebibliography}
 
\end{document}